\documentclass{article}

\usepackage{amsmath, amsfonts,amssymb}
\usepackage[dvips]{graphicx}
\usepackage{theorem, cases}
\usepackage{mathtools}

\newenvironment{proof}[1][Proof]{\textbf{#1.} }{\ \rule{0.5em}{0.5em}}

\newtheorem{thm}{Theorem}[section]

\newtheorem{cor}[thm]{Corollary}
\newtheorem{lem}[thm]{Lemma}
\newtheorem{prop}[thm]{Proposition}

\newtheorem{rmk}{Remark}[section]

{\theorembodyfont{\rmfamily}}
\newcommand{\fp}[2]{\frac{\partial #1}{\partial #2}}
\newcommand{\dint}{\displaystyle\int}

\newcommand{\be}{\begin{equation}}
\newcommand{\ee}{\end{equation}}

\newcommand{\bq}{\begin{eqnarray}}
\newcommand{\eq}{\end{eqnarray}}
\newcommand{\ex}{\mathbb{E}}

\newcommand{\bs}{\bigskip}

\newcommand{\ind}{1\hspace{-2.1mm}{1}} 
\newcommand{\nn}{\nonumber}

\newcommand{\ul}{\underline}
\newcommand{\ol}{\overline}

\newcommand{\pr}{\mathbb{P}}

\newcommand{\ep}{\mathbf{e}_q}
\newcommand{\epp}{\mathbf{e}_p}
\newcommand{\diff}{\textup{d}}

\begin{document}


\title{\Large \textbf{Occupation times, drawdowns, and drawups for one-dimensional regular diffusions}}

\author{
 Hongzhong Zhang\thanks{Dept. Statistics, Columbia University, New York, NY 10027 ({\tt hzhang@stat.columbia.edu})}
}
\maketitle
\begin{abstract}
The drawdown process of an one-dimensional regular diffusion process $X$  is given by $X$ reflected at its running maximum. The drawup process is given by $X$ reflected at its running minimum. We calculate the probability that a drawdown proceeds a drawup in an exponential time-horizon.  We then study the law of the occupation times of the drawdown process and the drawup process. These results are applied to address problems in risk analysis and for option pricing of the drawdown process. Finally, we present examples of  Brownian motion with drift and three-dimensional Bessel processes, where we prove an identity in law.
\end{abstract}



\section{Introduction}
In a filtered probability space $(\Omega, \mathbb{F}, \mathcal{F}, \pr)$ with filtration $\mathbb{F}=\{\mathcal{F}_t\}$, we consider a one-dimensional regular time-homogenous diffusion $X$ on $I:=(l,\infty)$ with natural (or entrance) boundaries. Its evolution is governed by the stochastic differential equation
\begin{eqnarray}\label{eq:diffusion}
 dX_t & = & \mu(X_t)\,\diff t+\sigma(X_t)\diff B_t,\,\,\,X_0=x \in I,
\end{eqnarray}
with an infinitesimal generator  
$$\mathcal{L}_X=\frac{1}{2}\sigma^2(x)\frac{\partial^2}{\partial x^2}+\mu(x)\fp{}{x}.$$ 
 Here $B$ is a standard $\mathbb{F}$-Brownian motion, and $(\mu(\cdot),\sigma(\cdot))$ is a pair of real-valued continuous functions.  We introduce the running maximum and minimum processes of $X$  by
 \[\ol{X}_t:=\sup_{0\le s\le t}X_s\,\text{ and }\,\ul{X}_t:=\inf_{0\le s\le t}X_s,\,\,\,t\ge0.\]
The drawdown process of $X$, and its dual, the drawup process, are then defined as  $Y:=\ol{X}-X$ and $\hat{Y}:=X-\ul{X}$, respectively.  The first passage times of $Y$ and $\hat{Y}$ above a positive threshold are  respectively called the drawdown  and drawup times.  The occupation time of a stochastic process is the amount of time the stochastic process stays within a certain range. 

Laplace transforms of stopping times and occupation times for diffusion processes are well studied due to analytical tractability and their applications in risk theory, mathematical finance, and engineering.  Some classical results for general diffusion processes using the Feynman-Kac representation and excursion theory can be found in \cite{PitmanYor1,PitmanYor2}.  Recent advances on this topic include jumps in the underlying model \cite{CaiChenWan,occupationLevy,LandPenaZhou},
or incorporate memory into the model \cite{FordePogudinZhang}.  Among all the stopping times studied, one that is especially interesting is the drawdown time, which finds applications in financial risk management \cite{CarrZhanHaji, GrosZhou, MagdAtiy, PospVece, Vece07}, the theory of optimal stopping \cite{Meil,Peskir, SheppShiryaev} and the problem of the quickest change-point detection \cite{HadjZhanPoor,PoorHadj,ShiryaevCUSUM,corr,HadjSIAMCO}.
 The Laplace transform of the drawdown time for general diffusion processes was first derived  in \cite{Lehoczky77}.
Recently, \cite{salmval} derived the probability that the minimum of the drawdown  and drawup times proceeds an  exponential random variable  for a standard
Brownian motion. In \cite{ZhanHadj12}, the authors derived the joint Laplace transform of the drawdown time and the so-called speed of market crash for a general diffusion process, using progressive enlargement of filtrations. Laplace transforms involving drawdowns for general spectrally negative L\'evy processes are studied, using excursion theory in \cite{MijatovicPistorius}, which generalizes a result for diffusion-type processes in \cite{HadjVece, PospVeceHadj, Zhang2010,ZhanHadj}.

In this paper, we obtain a class of results regarding the sequential order of the drawdown and  drawup times in a finite time horizon, and laws of occupation times  of the drawdown process $Y$ and the drawup process $\hat{Y}$,
for a general one-dimensional regular diffusion $X$. 
In particular, we   derive the probability that the drawdown time precedes the drawup time in an exponential time horizon. We then compute the Laplace transforms of the occupation times of $X$ below $y$ and below the starting point, until the first exit time and the drawdown time, respectively. Using these results, we proceed to study the law of occupation time of the drawdown process $Y$ above $y$, and of the drawup process $\hat{Y}$ below $y$, until the drawdown time or an exponential time independent of $X$.  

Rather than relying on standard techniques in excursion theory, we use a perturbation approach to obtain the Laplace transform results.  To the author's knowledge,  the first use of this approach dates back to \cite{Lehoczky77}, where the conditional Laplace transform of the drawdown time is obtained by an approximation argument based on the Laplace transforms for the first hitting times. Recent applications of this approach  can be seen in  \cite{occupationLevy, LandPenaZhou} for (refracted)  L\'evy process and \cite{LiZhou12} for diffusion process. In this work, we subsequently reduce the laws of complicated stopping times and occupation times at hand to those of the simpler. This reduction is made possible using argument with strong Markov property and progressive enlargement of filtration \cite{Protterbook, ZhanHadj12}.

The results obtained can be applied in risk analysis and for option pricing of the drawdown processes.  
 In particular, we consider a time-homogenous diffusion with reduced form default model. The probability of realizing a drawdown before a drawup before the default time can be computed when the hazard rate is a constant. Using the so called Omega model (see for example, \cite{AlbrGerbShiu, GerbShuiYang}), we can describe the hazard rate of the default in such a way that it depends on the asset process, its drawdown or drawup processes.  Then the probability of default before a large drawdown can be computed. Moreover, our results can be used to price Parisian-like digital call options  and $\alpha$-quantile options of the drawdown process, a non-trivial extension of the option pricing problem for maximum drawdowns \cite{CarrZhanHaji, PospVecePDE, Vece06, Vece07, ZhanLeunHadj}. 
 
As examples of our general result, we present explicit formulas for some of the main results in the cases of Brownian motions with drift and three-dimensional Bessel processes. Moreover, we prove through Laplace transform that, in these two models, the law of the occupation time of the drawdown process above a level is the same as that of the drawdown time of certain threshold.

The paper is structured in the following way: standard identities regarding the first hitting time and the first exit time of an one-dimensional regular diffusion are reviewed in Section 2. In Section 3, we derive the probability that a  drawdown precedes a drawup in a finite time-horizon. In Section 4, we use the perturbation method to derive the Laplace transform of the occupation time of $X$ below a level until the first exit time. We then use this result to compute various occupation time  formulas regarding the drawdown and drawup processes.  In Section 5,  we discuss applications in risk analysis and for option pricing of the results obtained in Sections 3 and 4.  In Section 6, we present explicit formulas for some of the main results in the cases of Brownian motion with drift and three-dimensional Bessel process. The proofs of propositions and theorems omitted can be found in the Appendix.
\section{Preliminaries}
\label{define}
Let $X$ be the linear diffusion process defined in \eqref{eq:diffusion}. The first hitting time and the first passage times to a level $y\in I$ by $X$ are respectively given by 
\bq
\tau_y&:=&\inf\{t>0\,:\,X_t=y\},\nn\\
\tau_y^{\pm}&:=&\inf\{t>0\,:\,X_t\gtreqqless y\}.\nn
\eq
Throughout the paper, we use $\pr_x(\cdot)$ to denote the measure $\pr(\cdot|X_0=x)$,  and $\ex_x$ the expectation under $\pr_x$. Moreover, for $q\ge0$, we denote by $\ep$ an exponential random variable\footnote{As a convention, we assume that $\pr(\mathbf{e}_0=\infty)=1$.} with parameter $q$, which is independent of $X$. 

It is well-known that, for any $q>0$, the {\em Sturm-Liouville equation} $(\mathcal{L}_Xf)(x)=q f(x)$ has a positive increasing solution $\phi_q^+(\cdot)$ (decreasing solution $\phi_q^-(\cdot)$, resp.). In fact, for an arbitrary fixed $\kappa\in I$, we can choose
\be\label{eq:phi}
{\phi_q^+(x)=}
\begin{dcases}\ex_x\{\textup{e}^{-q\tau_\kappa}\},\,\,\, &\text{if }x\le \kappa\\
\frac{1}{\ex_\kappa\{\textup{e}^{-q\tau_x}\}},\,\,\, &\text{if }x>\kappa
\end{dcases},\,\,\,\,\,
{\phi_q^-(x)=}\begin{dcases}
\frac{1}{\ex_\kappa\{\textup{e}^{-q\tau_x}\}},\,\,\, &\text{if }x\le \kappa\\
{\ex_x\{\textup{e}^{-q\tau_\kappa}\}},\,\,\, &\text{if }x>\kappa
\end{dcases},\,\,\,\forall x\in I.
\ee
A scale function of $X$, $s(\cdot)$, is an increasing function from $I$ to $\mathbb{R}$, such that $(\mathcal{L}_Xs)(x)=0$ for all $x
\in I$. In particular, one can choose $s'(x)=\exp(-\int_{\kappa'}^x\frac{2\mu(y)}{\sigma^2(y)}\diff y)$ for some $\kappa'\in I$.
Fix a scale function $s(\cdot)$,  there exists a constant $w_q>0$ such that (see, for example, page 19 of \cite{borodin2002handbook})
\be w_qs'(x)=(\phi_q^+)'(x)\phi_q^-(x)-(\phi_q^-)'(x)\phi_q^+(x).\label{scalefun}\ee
Furthermore, we define function 
\be
W_q(x,y):=w_q^{-1}\cdot\det\begin{bmatrix}\phi_q^+(x) & \phi_q^+(y)\\
\phi_q^-(x) &\phi_q^-(y)\end{bmatrix},\,\,\,\forall x,y\in I,
\ee
with its derivatives
\bq
W_{q,1}(x,y):=\frac{\partial}{\partial x}W_q(x,y),\,\,W_{q,2}(x,y):=\frac{\partial}{\partial y}W_{q,1}(x,y).
\eq
When $q=0$, we extend the definition of $W_q$ using
\be
W_0(x,y):=s(x)-s(y).\label{W0}
\ee
The functions $W_q$ and $\phi_q^-$ have the following properties:
\begin{lem}\label{property:W_q}
 For any $x,y,z\in I$, $q>0$
\begin{eqnarray}
\nn W_q(x,y)=-W_q(y,x),\,\,
\frac{\partial}{\partial x}\frac{W_q(x,y)}{W_q(x,z)}=\frac{s^{'}(x)}{W_q^2(x,z)}W_q(y,z),\,\,
\lim_{x\downarrow l}\phi_{q}^-(x)=\infty,\\
\lim_{q\downarrow 0}W_q(x,y)=W_0(x,y),\,\, \lim_{q\downarrow0}W_{q,1}(x,y)=s'(x).\nn
\end{eqnarray}
\end{lem}
\begin{proof}
The proof can be found  in the Appendix.
\end{proof}

We recall the following result regarding the first exit of $X$ from page 603 of \cite{Lehoczky77}.
 \begin{lem}\label{fundamentallemma}
Suppose that $x,y,z\in I$ and $(x-y)(z-x)>0$, for $q\ge0$, we have
\begin{eqnarray}
  \ex_x\{\textup{e}^{-q\tau_y}\,;\tau_y<\tau_z\}=\pr_x(\tau_y<\tau_z\wedge\ep)&=&\frac{W_q(x,z)}{W_q(y,z)}.
\end{eqnarray}

\end{lem}

\bs

\section{Drawdowns and drawups}
In this section we study the law of drawdowns and drawups of $X$, where $X$ is the linear diffusion process defined in \eqref{eq:diffusion}. To this end, we introduce the running maximum and minimum processes of $X$ by
\be
\ol{X}_t:=\sup_{o\le s\le t}X_s\,\text{ and }\, \ul{X}_t:=\inf_{0\le s\le t}X_s,\,\,t\ge0.\nn
\ee
The the drawdown and drawup processes of $X$ are defined respectively as
\begin{eqnarray}
Y_t:=\ol{X}_t-X_t\,\text{ and }\,
\hat{Y}_t:=X_t-\ul{X}_t,\,\,t\ge0,\nn
\end{eqnarray}
and the drawdown of $a$ units and the drawup of $b$ units are the following first passage times:
\begin{eqnarray}
\sigma_a&:=&\inf\{t\ge0,:\,Y_t\ge a\}, \,a>0,\\
\hat{\sigma}_b&:=&\inf\{t\ge 0\,:\,\hat{Y}_t\ge b\}, \,\,b>0.
\end{eqnarray}

In the remainder of this section, we will derive the probability that $\sigma_a$ proceeds $\hat{\sigma}_b$ in a finite time interval $[0,t]$ for any $t>0$. This is a nontrivial extension of the infinite time-horizon result in \cite{PospVeceHadj}, which enables us to study the drawdowns and drawups for diffusions with killing, as well as  limiting conditional distributions of $\sigma_a$ as the thresholds $a,b$ tend to infinity in a proper manner.  
The task is accomplished by computing the Laplace transform of the function $p_{x;a,b}(t):=\pr_x(\sigma_a<\hat{\sigma}_b\wedge t)$. 
In particular, for any $x, a$ such that $x-a\in I$, we will calculate:
\[\pr_x(\sigma_a<\hat{\sigma}_b\wedge\ep)=q\int_0^\infty \textup{e}^{-qt}\,\pr_x(\sigma_a<\hat{\sigma}_b\wedge t)\,\diff t.\]

\subsection{The case of $a\ge b>0$}
\begin{thm}\label{thm1}
On the event $\{\sigma_a<\hat{\sigma}_b\wedge\ep\}$, we have $\sigma_b<\hat{\sigma}_b$ and $X_{\sigma_b}\in(x-b,x)$. Moreover, for  $u\in (x-b,x)$ we have
\begin{multline}
\label{lequal}\pr_x(\sigma_a<\hat{\sigma}_b\wedge\ep, X_{\sigma_a}\in \,b-a+\diff u)\\=\frac{s^{'}(u+b)W_q(x,u)}{W_q^2(u+b,u)}\exp\bigg(\int_{u+b-a}^u\frac{W_{q,1}(v,v+b)}{W_q(v,v+b)}\diff v\bigg)\,\diff u.
\end{multline}
On the event $\{\hat{\sigma}_a<{\sigma}_b\wedge\ep\}$, we have $\hat{\sigma}_b<\sigma_b$ and $X_{\hat{\sigma}_b}\in (x,x+b)$. Moreover, for $u\in(x,x+b)$ we have,
  \begin{multline}\label{lapage}\pr_x(\hat{\sigma}_a<{\sigma}_b\wedge\ep, X_{\hat{\sigma}_a}\in a-b+\,\diff u)\\=\frac{s^{'}(u-b)W_q(u,x)}{W_q^2(u,u-b)} \exp\bigg(-\int_{u}^{u+a-b}\frac{W_{q,1}(v,v-b)}{W_q(v,v-b)}\,\diff v\bigg)\,\diff u.
\end{multline}

\end{thm}
\begin{proof}
First, it is easily seen that for $t>0$, and $u\in(x-b,x)$,
\begin{eqnarray}
\{\sigma_b\in \,\diff t, X_t\in \,\diff u, \hat{\sigma}_b>t\}&=&\{\tau_u\in \,\diff t, \ol{X}_t\in b+\diff u\},\,\,\,\pr_x\text{-a.s.}\nn
\end{eqnarray}
It follows from integration by parts that, 
\begin{eqnarray}\label{eq:prob}
\pr_x(\sigma_b<\hat{\sigma}_b\wedge\ep; X_{\sigma_b}\in \,\diff u)&=&\int_0^\infty q\textup{e}^{-qt}\pr_x(\sigma_b<\hat{\sigma}_b\wedge t, X_{\sigma_b}\in \,\diff u)\,\diff t\nn\\
&=&\int_0^\infty \textup{e}^{-qt}\pr_x(\sigma_b\in \,\diff t, X_t\in \,\diff u, \hat{\sigma}_b>t).\nn
\eq
On the other hand, we observe the fact that $\sigma_b\wedge\hat{\sigma}_b=\inf\{t\ge0\,:\,\ol{X}_t-\ul{X}_t\ge b\}$, $\pr_x$-a.s.,  is the first range time. From page 117 of \cite{borodin2002handbook}, we have, 
\bq
\pr_x(\sigma_b\in \,\diff t, X_t\in \diff u, \hat{\sigma}_b>t) & = & \frac{\partial}{\partial u}\pr_x\{\sigma_b\wedge\hat{\sigma}_b\in \diff t, X_t\le u\}\,\diff u\nn\\
&=&\frac{\partial}{\partial b}\pr_x(\tau_u\in \diff t, \tau_{u+b}>t)\,\diff u. \nn
\end{eqnarray}
From Lemmas \ref{property:W_q} and \ref{fundamentallemma}  we obtain  \eqref{lequal} for the case $a=b$. 

If $a>b$, then any path in the event $\{\sigma_a<\hat{\sigma}_b\wedge \ep\}$ can be decomposed into two path fragments: 
$\{X_t\}_{0\le t\le\sigma_b}$ and $\{X_{t}\}_{\sigma_b\le t\le\sigma_a}$. Conditioning on $\{X_{\sigma_b}=u\}$, the second path fragment is a process starting at $u$, and decreasing to $u+b-a$ before it incurs the drawup of $b$ units. Formally, using Markov shifting operator ($X_t\circ\theta(s)=X_{t+s}$), we have
\begin{eqnarray}
\sigma_a=\sigma_b+\tau_{u-a+b}^-\circ\theta(\sigma_b),\,\,\,\forall u\in(x-b,x).\nn
\end{eqnarray}
Using strong Markov property and memoryless of $\ep$, we have 
\begin{multline*}
\pr_x(\sigma_a<\hat{\sigma}_b\wedge\ep, X_{\sigma_a}\in b-a+\,\diff u)
\\
=\pr_x\{\sigma_b<\hat{\sigma}_b\wedge\ep,
X_{\sigma_b}\in \,\diff u\}\cdot \pr_u(\tau_{u-a+b}^-<\hat{\sigma}_b\wedge\ep).
\end{multline*}
Now \eqref{lequal} follows from Proposition \ref{propo} below. Eq. \eqref{lapage} can be proved using a similar argument.
\end{proof}

\begin{prop}\label{propo}
For $x,y\in I$, we define $m=\max(x,y)$ and $n=\min(x,y)$. Then we have
\begin{eqnarray}\label{hlap}
\pr_m(\tau_n^-<\hat{\sigma}_b\wedge\ep)&=&\exp\bigg(\int_n^m\frac{W_{q,1}(v,v+b)}{W_q(v,v+b)}\,\diff v\bigg),\\
\pr_n(\tau_m^+<\sigma_a\wedge\ep)&=&\exp\bigg(-\int_n^m\frac{W_{q,1}(v,v-a)}{W_q(v,v-a)}\,\diff v\bigg).\label{hlap1}
\eq
\end{prop}

\bs

\subsection{The case of $b>a>0$}
To obtain the result for general $b>a$, we recall the following result for the limiting case $b=\infty$ from page 602 of \cite{Lehoczky77}.
\begin{prop}\label{mdd}
\bq \pr_x(\sigma_a<\ep) =  \int_x^\infty\frac{s'(u)}{W_q(u,u-a)}\exp\bigg(-\int_x^u\frac{W_{q,1}(v,v-a)}{W_q(v,v-a)}\,\diff v\bigg)\,\diff u.
\eq
\end{prop}

A useful observation is that 
\be\label{de}
\pr_x(\sigma_a<\hat{\sigma}_b\wedge\ep)=\pr_x(\sigma_a<\ep)-\pr_x(\hat{\sigma}_b<\sigma_a<\ep).
\ee
The second term on the right hand side of the above equation is computed in the following result.
\begin{thm}\label{thm3}On the event $\{\hat{\sigma}_b<\sigma_a<\ep\}$, $X_{\hat{\sigma}_b}\in (x+b-a,x+b)$. Moreover, for $u\in (x,x+a)$ we have
\begin{multline}\label{thm3f}
\frac{\pr_x(\hat{\sigma}_b<\sigma_a<\ep, X_{\hat{\sigma}_b}\in b-a+\,\diff u)}{\,\diff u}\\
=\frac{s^{'}(u-a)W_q(u,x)}{W_q^2(u,u-a)} \exp\bigg(-\int_{u}^{u+b-a}\frac{W_{q,1}(v,v-a)}{W_q(v,v-a)}\,\diff v\bigg)\\\times\int_{u+b-a}^\infty\frac{s'(v)}{W_q(v,v-a)}\exp\bigg(-\int_{u+b-a}^v\frac{W_{q,1}(w,w-a)}{W_q(w,w-a)}dw\bigg)\,\diff v.
\end{multline}
\end{thm}
\begin{proof}
Notice that the event $\{\hat{\sigma}_b<\sigma_a<\ep\}=\{\hat{\sigma}_b<\sigma_a\wedge\ep\}\cap\{\sigma_a<\ep\}$, $\pr_x$-a.s. Using strong Markov property and memoryless property of $\ep$, we have
\begin{multline*}
\pr_x(\hat{\sigma}_b<\sigma_a<\ep, X_{\hat{\sigma}_b}\in b-a+\,\diff u)\,\\=\,\pr_x(\hat{\sigma}_b<\sigma_a\wedge\ep, X_{\hat{\sigma}_b}\in b-a+\,\diff u)\cdot\pr_{u+b-a}(\sigma_a<\ep).\nn
\end{multline*}
The result follows from Theorem \ref{thm1} and Proposition \ref{mdd}.
\end{proof}
\begin{rmk}
  By letting $q\to0^+$  in \eqref{lequal}, \eqref{lapage}, \eqref{de} and \eqref{thm3f}, and using \eqref{W0}, we obtain Theorems 4.1 and 4.2 in \cite{PospVeceHadj}. Second, by letting $b\to\infty$ in \eqref{de} and \eqref{thm3f}, one obtains the Laplace transform of $\sigma_a$.  The results in Theorems \ref{thm1} and \ref{thm3} also enables us to compute  $\ex_x\{\sigma_a|\sigma_a<\hat{\sigma}_b\}$ and the limit law of $\sigma_a$ given $\sigma_a<\hat{\sigma}_b$, as $a,b\to \infty$ in certain way. These distributional results can be applied to the problem of sequential detections and identification, which we  leave as future work.
\end{rmk}

\section{Occupation time formulas}
In this section, we begin by computing the Laplace transforms of the occupation time below a level until the first exit time for a linear diffusion process $X$ defined in \eqref{eq:diffusion}. Although relevant formulas exist for special diffusions, the results for general linear diffusions that we provide here are new.  Using this result, we then proceed to study occupation times of $X$, its drawdown $Y$ and its drawup $\hat{Y}$ until the drawdown time $\sigma_a$ or until an exponential time which is independent of $X$. These new results give several interesting  identities and provide means of measuring financial risk and pricing options as we shall in the next section.  
 
\subsection{Occupation time below a level until the first exit time}
In this subsection we study the law of occupation time until the first exit time. 
In particular, for $y\in (a,b)\subsetneq I$, the occupation time below $y$ before exiting $(a,b)$ is denoted by 
\be
A_{y}^{a,b}:=\int_0^{\tau_{a}^-\wedge\tau_b^+}\ind_{\{X_t<y\}}\,\diff t.\ee
The law of $A_y^{a,b}$ is summarized in the following result, the proof of which can be found in the Appendix.
\begin{prop}\label{thm:occupation_exit} 
For $x,y\in(a,b)\subsetneq I$, $q>0, p\ge0$, we have
\begin{multline}\label{eq:oct_up}
{\ex_x\{\textup{e}^{-qA_y^{a,b}-p\tau_b^+}; \tau_b^+<\tau_a^-\}}\\\quad=\begin{dcases}
\frac{W_{q+p}(x,a)}{W_{q+p}(y,a)}\frac{s'(y)}{W_{p,1}(y,b)+W_p(b,y)\dfrac{W_{q+p,1}(y,a)}{W_{q+p}(y,a)}}, \,\,\,&\text{if }{x\in(a,y]}\\
\frac{W_p(x,y)}{W_p(b,y)}+\frac{W_p(b,x)}{W_p(b,y)}\frac{s'(y)}{W_{p,1}(y,b)+W_p(b,y)\dfrac{W_{q+p,1}(y,a)}{W_{q+p}(y,a)}},\,\,\,&\text{if }x\in(y,b)
\end{dcases};
\end{multline}
\begin{multline}\label{eq:oct_down}
{\ex_x\{\textup{e}^{-qA_y^{a,b}}; \tau_b^+>\tau_a^-\}}\\\quad=\begin{dcases}
\dfrac{W_q(y,x)}{W_q(y,a)}+\dfrac{W_q(x,a)}{W_q(y,a)}\dfrac{{s'(y)}}{{W_{q,1}(y,a)}+\dfrac{s'(y)}{s(b)-s(y)}W_q(y,a)},\,\,\, &\text{if }x\in(a,y]\\
\dfrac{(s(b)-s(x)){s'(y)}}{(s(b)-s(y)){W_{q,1}(y,a)}+{s'(y)}W_q(y,a)},\,\,\, &\text{if }x\in(y,b)
\end{dcases}.
\end{multline}

\end{prop}
\begin{proof}
The proof can be found  in the Appendix.
\end{proof}

\begin{rmk}
Lemma \ref{fundamentallemma} can be recovered from Proposition \ref{thm:occupation_exit} by the limit $y\uparrow b$ in \eqref{eq:oct_up} and \eqref{eq:oct_down}. Moreover, Corollary 3.1 in \cite{LiZhou12} can be easily obtained by letting $x=y$, $b\uparrow\infty$ and $a\downarrow l$.
\end{rmk}

Letting $a\downarrow l$ in \eqref{eq:oct_up} and using Lemma \ref{property:W_q}, we obtain the following result.
\begin{cor}\label{cor:onesided}
For $x\in(y,b)\subsetneq I$, $p>0, q\ge0$, we have
\bq
\nn&&\ex_x\{\exp\bigg(-q\int_0^{\tau_b^+}\ind_{\{X_t<y\}}\,\diff t-p\tau_b^+\bigg)\}\\
\nn&=&\frac{W_p(x,y)}{W_p(b,y)}+\frac{W_p(b,x)}{W_p(b,y)}\dfrac{s'(y)}{W_{p,1}(y,b)+W_p(b,y)\dfrac{{\phi_{q+p}^{+}}'(y)}{\phi_{q+p}^+(y)}}.
\eq
\end{cor}

\subsection{Occupation time below a level until the drawdown time}
\bs
For $x-a, y\in I$, the occupation time below $y$ before the drawdown process $Y$ hits $a$ is denoted by
\be
B_y^a\,:=\,\int_0^{\sigma_a}\ind_{\{X_t<y\}}\,\diff t.
\ee
While the occupation time of $X$ below $y$ until an exponential random variable independent of $X$ is well-studied in literature \cite{LiZhou12}, the new quantity $B_y^a$ defined as above relates the occupation time of $X$ to its drawdown process $Y$, which can be used to characterize the drawdown risk of $X$.
The law of $B_y^a$ is summarized in the following results, the proof of which can be found in the Appendix.
\begin{prop}\label{thm:otbelow}
For $q\ge0$ and $x\in I$,
\begin{multline*}
\ex_x\{\textup{e}^{-qB_x^a}\}=\exp\bigg(-\int_x^{x+a}\frac{\frac{s'(u)}{s'(x)}\frac{W_{q,1}(x,u-a)}{W_q(x,u-a)}\,\diff u}{1+\frac{s(u)-s(x)}{s'(x)}\frac{W_{q,1}(x,u-a)}{W_q(x,u-a)}}\bigg)\\+\int_x^{x+a}\frac{1}{\left(1+\frac{s(u)-s(x)}{s'(x)}\frac{W_{q,1}(x,u-a)}{W_q(x,u-a)}\right)^2}\frac{s'(u)\,\diff u}{W_q(x,u-a)}.
\end{multline*} 
\end{prop}
\begin{proof}
The proof can be found  in the Appendix.
\end{proof}

\subsection{Occupation time of the drawdown process until the drawdown time}
For any $y\in (0,a)$ and $x-a\in I$, the occupation time of the drawdown process $Y$ above $y$ before $Y$ hits $a$ is denoted by
\be
C_y^a\,:=\,\int_0^{\sigma_a}\ind_{\{Y_t>y\}}\,\diff t.\label{Cy}
\ee
The occupation time $C_y^a$ measures the amount of time  for the drawdown process $Y$ to finish the ``last trip'' from $y$ to $a$. It can be used as a measurement of performance for CUSUM-type stopping rule in change-point detection problems \cite{PoorHadj}.
Because of its obvious financial interpretation, $C_y^a$ can also be used as a measure for drawdown risks.

To obtain the law of $C_y^a$, we first condition on $\ol{X}_{\sigma_a}$, and then count separately the occupation time before and after the moment when the peak $\ol{X}_{\sigma_a}$ is realized. To this end, we recall the following result from page 602 of \cite{Lehoczky77}.
\begin{lem}\label{dmax}For $m>x$,
\be
\pr_{x}(\ol{X}_{\sigma_a}\ge m)=\exp\bigg(-\int_x^m\frac{s'(v)}{s(v)-s(v-a)}\,\diff v\bigg).
\ee
\end{lem}

\begin{thm} \label{thm:otdd}For $q\ge 0$, $0<y<a$, 
\begin{multline*}
\ex_x\{\textup{e}^{-qC_y^a}\,;\,\sigma_a<\infty\}=\int_x^\infty \dfrac{\frac{s'(m)}{W_q(m-y,m-a)}}{1+\frac{s(m)-s(m-y)}{s'(m-y)}\frac{W_{q,1}(m-y,m-a)}{W_q(m-y,m-a)}}\\
\times\exp\bigg(-\int_x^m\dfrac{\frac{s'(u)}{s'(u-y)}\frac{W_{q,1}(u-y,u-a)}{W_q(u-y,u-a)}\diff u}{1+\frac{s(u)-s(u-y)}{s'(u-y)}\frac{W_{q,1}(u-y,u-a)}{W_q(u-y,u-a)}}\bigg)\,\diff m.
\end{multline*}
\end{thm}
\begin{proof}
Following the conditioning argument in \cite{ZhanHadj12}, we define 
\be
g_a:=\sup\{t\le\sigma_a\,:\,X_t=\ol{X}_t\}.
\ee
Given $X_{g_a}$, the path fragments $\{X_t\}_{t\in[0,g_a]}$ and $\{X_t\}_{t\in[g_a,\sigma_a]}$ are two independent conditional processes.  Moreover,  by Proposition 1 of \cite{ZhanHadj12}, the optional projection of non-increasing process $\ind_{\{g_a>t\}}$:
\[\chi_t\mathop{:=}\pr_x(g_a>t|\mathcal{F}_t),\]
is a supermartingale, with a Doob-Meyer decomposition  
\[\chi_t=M_t-L_t,\]
where 
\[M_t=1+\int_0^{t\wedge\sigma_a}\frac{s'(X_s)\sigma(X_s)\diff B_s}{s(\ol{X}_s)-s(\ol{X}_s-a)},\,\,\,L_t=\int_0^{t\wedge\sigma_a}\frac{s'(\ol{X}_s)\diff \ol{X}_s}{s(\ol{X}_s)-s(\ol{X}_s-a)}.\]
Now introduce a nonnegative bounded optional process 
\be
\Gamma_t=\exp\bigg(-q\int_0^t\ind_{\{Y_s>y\}}\diff s\bigg)\ind_{\{t<\sigma_a<\infty\}},\,\,\,t\ge0.\nn
\ee 
Using the same argument as in the proof of Theorem 15 on page 380 of \cite{Protterbook}, we have that, for any positive test function $f(\cdot)$ on $[0,\infty)$, 
\[\ex_x\{f(X_{g_a})\Gamma_{g_a}\}\,=\,\ex_x\{\int_0^\infty f(X_t)\Gamma_t\diff L_t\}\,=\,\ex_x\{\int_0^\infty  \frac{f(X_t)\Gamma_t\cdot s'(\ol{X}_t)\diff \ol{X}_t}{s(\ol{X}_t)-s(\ol{X}_t-a)}\}.\] 
By using a change of variable, $m=\ol{X}_t$ in the above equation, and also the fact that $X_t=\ol{X}_t$ on the support of measure $\diff \ol{X}_t$, we have that 
\be
\ex_x\{f(X_{g_a})\Gamma_{g_a}\}\,=\,\int_x^\infty f(m)\ex_x\{\exp\bigg(-q\int_0^{\tau_{m}^+}\ind_{\{Y_t>y\}}\diff t\bigg)\ind_{\{\tau_m^+<\sigma_a\}}\}\frac{s'(m)\diff m}{s(m)-s(m-a)}.\label{eq:fGamma}
\ee 
On the other hand, from Lemma \ref{dmax} we have that, for all $u>x$,  
\bq\label{eq:conditionalevent}
\pr_x(\tau_m^+<\sigma_a)&=&\exp\bigg(-\int_x^m\frac{s'(v)}{s(v)-s(v-a)}\diff v\bigg),\\
\pr_x(X_{g_a}\in \diff m)&=&\frac{s'(m)}{s(m)-s(m-a)}\exp\bigg(-\int_x^m\frac{s'(v)}{s(v)-s(v-a)}\diff v\bigg).\label{eq:densitymax}
\eq
From \eqref{eq:fGamma} and \eqref{eq:conditionalevent} we have,
\begin{multline}
\ex_x\{f(X_{g_a})\Gamma_{g_a}\}=\int_x^\infty f(m) \cdot\ex_x\{\exp\bigg(-q\int_0^{\tau_m^+}\ind_{\{Y_t>y\}}\diff t\bigg)|\tau_m^+<\sigma_a\}\\
\times \frac{s'(m)}{s(m)-s(m-a)}\exp\bigg(-\int_x^m\frac{s'(v)}{s(v)-s(v-a)}\diff v\bigg)\diff m\\
=\int_x^\infty f(m) \cdot\ex_x\{\exp\bigg(-q\int_0^{\tau_m^+}\ind_{\{Y_t>y\}}\diff t\bigg)|\tau_m^+<\sigma_a\}\cdot\pr_x(X_{g_a}\in \diff m),\label{eq:jointint}
\end{multline}
where the last line follows from \eqref{eq:densitymax}. It follows from \eqref{eq:jointint} that, 
\[\ex_x\{\exp\bigg(-q\int_0^{\tau_m^+}\ind_{\{Y_t>y\}}\diff t\bigg)|\tau_m^+<\sigma_a\}=\ex_x\{\exp\bigg(-q\int_0^{g_a}\ind_{\{Y_t>y\}}\diff t\bigg)|X_{g_a}=m\}.\]
To get the conditional expectation on the left hand side for any $m>x$, we let $\epsilon=\frac{m-x}{N}$ for a large integer $N>0$.  Using Lebesgue dominated convergence theorem, continuity and strong Markov property of $X$, we have that 
\bq
&&\ex_x\{\exp\bigg(-q\int_0^{\tau_m}\ind_{\{Y_t>y\}}\,\diff t\bigg)\,|\,\sigma_a>\tau_m^+\}\nn\\
&=&\lim_{N\to\infty}\prod_{i=0}^{N-1}\ex_{x+i\epsilon}\{\textup{e}^{-qA_{x+i\epsilon-y}^{x+i\epsilon-a, x+(i+1)\epsilon}}\,|\,\tau_{x+(i+1)\epsilon}^+<\tau_{x+i\epsilon-a}^-\}\nn\\
&=&\lim_{N\to\infty}\exp\bigg(\log\bigg[\sum_{i=0}^{N-1}\ex_{x+i\epsilon}\{\textup{e}^{-qA_{x+i\epsilon-y}^{x+i\epsilon-a, x+(i+1)\epsilon}}\,|\,\tau_{x+(i+1)\epsilon}^+<\tau_{x+i\epsilon-a}^-\}\bigg]\bigg)\nn\\
&=&\exp\bigg(\lim_{N\to\infty}\sum_{i=0}^{N-1}\left[\ex_{x+i\epsilon}\{\textup{e}^{-qA_{x+i\epsilon-y}^{x+i\epsilon-a, x+(i+1)\epsilon}}\,|\,\tau_{x+(i+1)\epsilon}^+<\tau_{x+i\epsilon-a}^-\}-1\right]\bigg)\nn\\
&=&\exp\bigg(\int_x^m\bigg[\frac{s'(u)}{s(u)-s(u-a)}-\dfrac{\frac{s'(u)}{s'(u-y)}\frac{W_{q,1}(u-y,u-a)}{W_q(u-y,u-a)}}{1+\frac{s(u)-s(u-y)}{s'(u-y)}\frac{W_{q,1}(u-y,u-a)}{W_q(u-y,u-a)}}\bigg]\,\diff u\bigg),\nn
\eq
where we used Proposition \ref{thm:occupation_exit} in the last equality.
Similarly, for the occupation time after the random time $g_a$, we have that  
\bq
&&\ex_x\{\exp\bigg(-q\int_{g_a}^{\sigma_a}\ind_{\{Y_t>y\}}\,\diff t\bigg)\,|\,X_{g_a}=m\}\nn\\&=&\ex_m\{\exp\bigg(-q\int_0^{\tau_{m-a}^-}\ind_{\{X_t<m-y\}}\,\diff t\bigg)\,|\,\tau_{m-a}^-<\tau_m^+\}\nn\\
&=&\lim_{\epsilon'\to0+}\dfrac{\ex_m\{\textup{e}^{-qA_{m-y}^{m-a,m+\epsilon'}}\,;\, \tau_{m-a}^-<\tau_{m+\epsilon'}^+\}}{\pr_m(\tau_{m-a}^-<\tau_{m+\epsilon'}^+)}\nn\\
&=&\dfrac{\frac{s(m)-s(m-a)}{W_q(m-y,m-a)}}{1+\frac{s(m)-s(m-y)}{s'(m-y)}\frac{W_{q,1}(m-y,m-a)}{W_q(m-y,m-a)}}.\nn
\eq
The result now follows from integration using the density in \eqref{eq:densitymax}.
\end{proof}

\subsection{Occupation time of the drawup process until the drawdown time}
For any $y\in[a,\infty)$ and $x-a\in I$, the occupation time of the drawup process $\hat{Y}$ below $y$ before the drawdown process $Y$ hits $a$ is denoted by
\be
D_y^a\,:=\,\int_0^{\sigma_a}\ind_{\{\hat{Y}_t<y\}}\,\diff t.
\ee
The occupation time $D_y^a$ can be considered as a counterpart of $C_y^a$ defined in \eqref{Cy}.
The law of $D_y^a$ is summarized in the following result. 
\begin{thm}\label{thm:otdu}
For all $q>0$ and $y\ge a$,
\begin{multline*}
\ex_x\{\textup{e}^{-qD_y^a}; \sigma_a<\infty\}=\pr_x(\sigma_a<\hat{\sigma}_y\wedge\ep)+\int_{x}^{x+a}\frac{s'(u-a)W_q(u,x)}{W_q^2(u,u-a)}\\
\times\exp\bigg(-\int_u^{u+y-a}\frac{W_{q,1}(v,v-a)}{W_q(v,v-a)}\,\diff v\bigg)\cdot\ex_{u+y-a}\{\textup{e}^{-qB_{u+y-a}^a}\}\,\diff u,\,\,\,\,\,
\end{multline*}
where the probability in the first line is given in \eqref{de}, and the expectation in the last line is given in Proposition \ref{thm:otbelow}.
\end{thm}
\begin{proof}
Notice that $D_y^a=\sigma_a$, $\pr_x$-a.s. on the event $\{\sigma_a<\hat{\sigma}_y\}$. On the other hand, on the event $\{\hat{\sigma}_y<\sigma_a\}$, we have $\ul{X}_{\hat{\sigma}_y}=X_{\hat{\sigma}_y}-y\le X_{\hat{\sigma}_y}-a<\ol{X}_{\sigma_a}-a$, $\pr_x$-a.s. 
Thus,
\[D_y^a=\hat{\sigma}_y+B_{X_{\hat{\sigma}_y}}^a\circ\theta(\hat{\sigma}_y).\]
Using strong Markov property we have,
\bq
&&\ex_x\{\textup{e}^{-qD_y^a}; \sigma_a<\infty\}\nn\\
&=&\ex_x\{\textup{e}^{-q\sigma_a}; \sigma_a<\hat{\sigma}_y\}+\ex_x\{\textup{e}^{-q\hat{\sigma}_y}\ind_{\{\hat{\sigma}_y<\sigma_a\}}\ex_{X_{\hat{\sigma}_y}}\{\textup{e}^{-qB_{X_{\hat{\sigma}_y}}^a}\}\}\nn\\
&=&\pr_x(\sigma_a<\hat{\sigma}_y\wedge\ep)+\int_x^{x+a}\pr_x(\hat{\sigma}_y<\sigma_a\wedge\ep, X_{\hat{\sigma}_y}\in y-a+\,\diff u)\nn\\
&&\times\ex_{u-a+y}\{\textup{e}^{-qB_{u-a+y}^a}\}.\nn
\eq
The result now follows from Theorem \ref{thm1}.
\end{proof}
\subsection{Occupation time of the drawdown process at an independent exponential time}
For any $y\in(0,\infty)$ and $x-y\in I$, the occupation time of the drawdown process $Y$ above $y$ before an independent exponential time $\ep$ is denoted by
\be
E_y^q\,:=\,\int_0^{\ep}\ind_{\{Y_t>y\}}\,\diff t.
\ee
The occupation time $E_y^q$ can be identified with the Laplace transform of the occupation time in a finite time horizon $T>0$:
\[\int_0^T\ind_{\{Y_t>y\}}\diff t,\]
which is closely related to the maximum drawdown $\sup_{t\in[0,T]}Y_t$ and other quantiles of $Y$ in the finite time-horizon $T$.
The law of $E_y^q$ is summarized in the following result, the proof of which can be found in the Appendix.
\begin{thm}\label{thm:ddexp}
For all $q,p>0$ and $y>0$, 
\begin{multline}
\ex_x\{\textup{e}^{-pE_y^q}\}=1-\exp\bigg(-\int_x^\infty\frac{W_{q,2}(u-y,u)+W_{q,1}(u,u-y)\frac{{\phi_{q+p}^+}'(u-y)}{\phi_{q+p}^+(u-y)}}{W_{q,1}(u-y,u)+W_q(u,u-y)\frac{{\phi_{q+p}^+}'(u-y)}{\phi_{q+p}^+(u-y)}}\,\diff u\bigg)\\
-\int_x^\infty\exp\bigg(-\int_x^m\frac{W_{q,2}(u-y,u)+W_{q,1}(u,u-y)\frac{{\phi_{q+p}^+}'(u-y)}{\phi_{q+p}^+(u-y)}}{W_{q,1}(u-y,u)+W_q(u,u-y)\frac{{\phi_{q+p}^+}'(u-y)}{\phi_{q+p}^+(u-y)}}\,\diff u\bigg)\\
\times\frac{\frac{p}{q+p}s'(m)\frac{{\phi_{q+p}^+}'(m-y)}{\phi_{q+p}^+(m-y)}}{W_{q,1}(m-y,m)+W_q(m,m-y)\frac{{\phi_{q+p}^+}'(m-y)}{\phi_{q+p}^+(m-y)}}\,\diff m.
\end{multline}
\end{thm}
\begin{proof}
The proof can be found  in the Appendix.
\end{proof}
\section{Application}
\subsection{Probabilities regarding drawdowns and defaults}
A realization of a large drawdown is usually considered to be a sign of market recession. In this section, we use a reduced form model for default and compute the probabilities of  drawdowns and default. In particular,  
we consider an asset process $X$, which is given by a time-homogeneous diffusion process with initial value $x$ and lifetime $\zeta$:
\be
dX_t=\mu(X_t)\diff t+\sigma(X_t)\diff B_t,\,\, t<\zeta.
\ee 
Here $\zeta$ is an independent positive random variable which models the ``default time'' of asset process $X$. If we assume that the ``default time'' $\zeta=\ep$ for a $q>0$, then the probability that there is a drawdown of $a$ units before a drawup of $b$ units by the default time is given by $\pr_x(\sigma_a<\hat{\sigma}_b\wedge \ep),$ which is readily available from Theorems \ref{thm1} and \ref{thm3}. 

Moreover, we can consider a more realistic model for the default time $\zeta$, to reflect the fact that realizations of drawdowns of the asset process are very likely to be followed by a default. In particular, we adopt the Omega model studied in \cite{AlbrGerbShiu, GerbShuiYang, LandPenaZhou} to model the hazard rate of $\zeta$ at time $t>0$ as $q\ind_{\{Y_t>y\}}$:
\be
\pr_x(\zeta\in t+\diff t|\zeta>t)=q\ind_{\{Y_t>y\}}\diff t.
\ee
for some $y\in (0,a)$. 
 Here $a$ is a large number that characterize a critical level of a large drawdown. Then the probability of default before the drawdown of $a$ units is given by 
\be
\pr_x(\zeta<\sigma_a)\,=\,1-\pr_x(\zeta\ge\sigma_a)=1-\ex_x\{\textup{e}^{-qC_y^a}\}.\label{eq:default}
\ee 
The Laplace transform in \eqref{eq:default} can be found in Theorem \ref{thm:otdd}.

Hazard rate of similar form can be considered. For example, we can model the hazard rate of the default time $\zeta$ as $q\ind_{\{X_t<x\}}$. Here the initial value $x$ is a critical benchmark level which may trigger a default through default intensity. We can also model the hazard rate of the  default time $\zeta$ as $q\ind_{\{\hat{Y}_t<y\}}$. This is the case in which the default tend to occur when there is not enough upside momentum for the asset process. In both cases, the probability of default before a drawdown of $a$ units can be found using Proposition \ref{thm:otbelow} and \ref{thm:otdu}.

\subsection{Option pricing for the drawdown process}
Options on maximum drawdown and drawdown processes have drawn lots of attentions in recent years (see \cite{CarrZhanHaji, CherNikePlat, PospVecePDE, Vece06, yamasatotaka, ZhanLeunHadj}). In this section, we use an semi-analytic approach to price a large class of options on the drawdown process. In particular, we assume that the market is complete, $\pr$ is the risk-neutral measure, and $r\ge 0$ is the risk-free interest rate. We model the underlying process\footnote{Notice that the underlying process is not necessarily an asset price process. It can be, for example, the logarithm of an asset price process.}  as the time-homogeneous diffusion $X$ defined in \eqref{eq:diffusion}. Then a Parisian-like (see \cite{ChesPicqYor} for a definition of standard Parisian option) digital call on the drawdown process with barrier $y>0$, maturity $T>0$ and strike $K\in (0,T)$ is worth 
\be\label{eq:parisianprice}
P_0(x, y, K, T)\,=\,e^{-rT}\pr_x(\int_0^T\ind_{\{Y_t>y\}}\,\diff t>K),
\ee
at its inception. Using double randomizations: $T=\ep$ and $K=\epp$, then the randomized option price is given by 
\be
\hat{P}_0(x, y, q, p)\,=\,e^{-rT}-e^{-rT}\ex_x\{e^{-pE_y^q}\}.\label{eq:parisian}
\ee 
The Laplace transform in \eqref{eq:parisian} is readily available in Theorem \ref{thm:ddexp}.  Hence, the price \eqref{eq:parisianprice} can be computed via double Laplace inversion:
\be
P_0(x, y, K,T)\,=\,e^{-rT}-e^{-rT}\cdot\mathcal{I}_p\bigg(\frac{1}{p}\cdot\mathcal{I}_q\bigg(\frac{1}{q}\ex_x\{e^{-pE_y^q}\}\bigg)\bigg|_{T}\bigg)\bigg|_{K},
\ee
where $\mathcal{I}_p$ and $\mathcal{I}_q$ are Laplace inversion operators. 

Corridor options such as $\alpha$-quantile option are studied in \cite{miuraOrderStatistics,miura}. Below we consider pricing of $\alpha$-quantile options on the drawdown process. To this end, for an $\alpha\in(0,1]$, we define the $\alpha$-quantile of the drawdown process during $[0,T]$ by 
\be
Y_T^\alpha\,:=\,\inf\{y>0\,:\,\int_0^T\ind_{\{Y_t>y\}}\, \diff t\le(1-\alpha)T.\}
\ee
In particular, $Y_T^1=\sup_{t\in[0,T]}Y_t$ is the maximum drawdown at time $T$. An option on the $\alpha$-quantile with maturity $T>0$ and an absolute continuous, bounded payoff function $f(\cdot)$ such that $f(0)=0$ is worth
\be
A_0(x,f,T)\,=\,e^{-rT}\ex_x\{f(Y_T^\alpha)\},
\ee
at its inception. We notice that
\bq
\ex_x\{f(Y_T^\alpha)\}&=&\ex_x\{\int_0^\infty\ind_{\{Y_T^\alpha\ge u\}}f'(u)\,\diff u\}=\int_0^\infty f'(u)\cdot\pr_x(Y_T^\alpha> u)\,\diff u\nn
\\&
=&\int_0^\infty f'(u)\cdot\pr_x(\int_0^T\ind_{\{Y_t>u\}}\diff t\ge(1-\alpha)T)\,\diff u.\nn
\eq
It follows from \eqref{eq:parisianprice} that
\be
A_0(x,f,T)\,=\,\int_0^\infty f'(u)\cdot P_0(x,u,(1-\alpha)T,T)\,\diff u. \label{eq:alphaquantile}
\ee 
Again, by double Laplace inversion and Theorem \ref{thm:ddexp} we can compute the price in \eqref{eq:alphaquantile}.

\section{Examples}
\subsection{Brownian motion with drift}
In this section we derive a group of explicit formulas for a Brownian motion with drift. In particular, we
consider a  Brownian motion with drift $\mu\neq0$ and diffusion coefficient $\sigma>0$:
\[dX_t=\mu \,\diff t+\sigma \diff B_t,\,\,\, I=(-\infty,\infty).\]
 Let us denote by 
\be 
\delta:=\frac{\mu}{\sigma^2},\,\,\,\, \gamma:=\sqrt{\delta^2+\frac{2q}{\sigma^2}}.\nn
\ee
Then the increasing and the decreasing eigenfunctions of $X$ can be chosen as (see for example, \cite{borodin2002handbook})
\be
\phi_q^+(x)=\textup{e}^{(\gamma-\delta)x}, \,\,\phi_q^-(x)=\textup{e}^{-(\gamma+\delta)x}.\nn
\ee
Fix the scale function $s(x)=\frac{1}{\delta}(1-\mathrm{e}^{-2\delta x})$, we have that
\bq
w_q={\gamma},\,\,W_q(x,y)=2\textup{e}^{-\delta(x+y)}\frac{\sinh[\gamma(x-y)]}{\gamma},\,\,
\frac{W_{q,1}(x,y)}{W_q(x,y)}=\gamma\coth[\gamma(x-y)]-\delta.\nn
\eq
From Theorems \ref{thm1} and \ref{thm3} we have that:
\begin{cor}\label{cor:dddu}
\begin{multline}
{\pr_0(\sigma_a<\hat{\sigma}_b\wedge\ep)}\\\,=\begin{dcases}
\frac{\sigma^2\gamma}{2q}\bigg(\frac{\textup{e}^{-\delta b}(\gamma\coth[\gamma b]+\delta)}{\sinh[\gamma b]}-\frac{\gamma}{\sinh^2[\gamma b]}\bigg)\textup{e}^{-(a-b)\left(\delta+\gamma\coth[\gamma b]\right)},&a\ge b>0;\nn\\
\dfrac{1-\frac{\sigma^2\gamma}{2q}\bigg(\frac{\textup{e}^{\delta a}(\gamma\coth[\gamma a]-\delta)}{\sinh[\gamma a]}-\frac{\gamma}{\sinh^2[\gamma a]}\bigg)\textup{e}^{-(b-a)\left(-\delta+\gamma\coth[\gamma a]\right)}}{\gamma\cosh[\gamma a]-\delta\sinh[\gamma a]}{\gamma \textup{e}^{-\delta a}}\nn, \,& b>a>0.
\end{dcases}
\end{multline}
\begin{multline}
{\pr_0(\hat{\sigma}_b<{\sigma}_a\wedge\ep)}\\
\,=\begin{dcases}
\frac{\sigma^2\gamma}{2q}\bigg(\frac{\textup{e}^{\delta a}(\gamma\coth[\gamma a]-\delta)}{\sinh[\gamma a]}-\frac{\gamma}{\sinh^2[\gamma a]}\bigg)\textup{e}^{-(b-a)\left(-\delta+\gamma\coth[\gamma a]\right)},\,\,&b\ge a>0;\nn\\
\dfrac{1-\frac{\sigma^2\gamma}{2q}\bigg(\frac{\textup{e}^{-\delta b}(\gamma\coth[\gamma b]+\delta)}{\sinh[\gamma b]}-\frac{\gamma}{\sinh^2[\gamma b]}\bigg)\textup{e}^{-(a-b)\left(\delta+\gamma\coth[\gamma b]\right)}}{\gamma\cosh[\gamma b]+\delta\sinh[\gamma b]}
{\gamma \textup{e}^{\delta b}}\nn, \,\, &a>b>0.
\end{dcases}
\end{multline}
\end{cor}
%

For occupation time of the drawdown process, from Theorem \ref{thm:otdd} we have that:
\begin{cor}\label{thm:samelaw}
For any $y\in(0,a)$, the occupation time $C_y^a$ has the same law as the  drawdown time $\sigma_{a-y}$.
\end{cor}
\begin{proof}
Straightforward calculation yields that
Using Theorem \ref{thm:otdd} we obtain that
\bq
\ex_0\{\textup{e}^{-qC_y^a}\}&=&\int_0^\infty \dfrac{\frac{s'(m)}{W_q(m-y,m-a)}}{1+\frac{s(m)-s(m-y)}{s'(m-y)}\frac{W_{q,1}(m-y,m-a)}{W_q(m-y,m-a)}}\nn\\
&&\times\exp\bigg(-\int_0^m\dfrac{\frac{s'(u)}{s'(u-y)}\frac{W_{q,1}(u-y,u-a)}{W_q(u-y,u-a)}}{1+\frac{s(u)-s(u-y)}{s'(u-y)}\frac{W_{q,1}(u-y,u-a)}{W_q(u-y,u-a)}}\bigg)\,\diff m\nn\\
&=&\frac{\gamma \textup{e}^{-\delta(a-y)}}{\gamma\cosh[\gamma(a-y)]-\delta\sinh[\gamma(a-y)]}=\ex_0\{\textup{e}^{-q\sigma_{a-y}}\},\nn
\eq
where the last equality follows from Proposition \ref{mdd} or Corollary \ref{cor:dddu} as $b\to\infty$. It follows that the occupation time $C_y^a$ has the same distribution as $\sigma_{a-y}$ under $\pr_0$. 
\end{proof}

\subsection{Three-dimensional Bessel process (BES(3))}
In this section we study the case of three-dimensional Bessel process. In particular, we consider 
\[dX_t=\frac{1}{X_t}\,\diff t+\diff B_t,\,\,\, I=(0,\infty).\]
Let us denote by 
\be
\nu:=\sqrt{2q}.\nn
\ee
Then the increasing and the decreasing eigenfunctions of $X$ can be chosen as  (see for example, \cite{borodin2002handbook})
\be
\phi_q^+(x)=\frac{1}{x}\frac{\sinh[\nu x]}{\sinh[\nu]},\,\,\,\phi_q^-(x)=\frac{\textup{e}^{-\nu (x-1)}}{x}.\nn
\ee
Fix the scale function $s(x)=-\frac{1}{x}$, we have that:
\bq
w_q=\frac{\nu\mathrm{e}^{\nu}}{\sinh(\nu)},\,\,
W_q(x,y)=\frac{1}{\nu x y}\sinh[\nu(x-y)],\,\,
\frac{W_{q,1}(x,y)}{W_q(x,y)}=-\frac{1}{x}+\nu\coth[\nu(x-y)].\nn
\eq
Using Theorem \ref{thm:otdd} we have that 
\begin{cor}\label{thm:samelaw1}
For $x>a>y>0$, the law of the occupation time $C_y^a$ is the same as the drawdown time $\sigma_{a-y}$.
\end{cor}
\begin{proof}
Straightforward calculation yields that 
Using Theorem \ref{thm:otdd} we obtain that
\begin{multline}\label{bes3}
\ex_x\{\textup{e}^{-qC_y^a}\}=\int_x^\infty \dfrac{\frac{s'(m)}{W_q(m-y,m-a)}}{1+\frac{s(m)-s(m-y)}{s'(m-y)}\frac{W_{q,1}(m-y,m-a)}{W_q(m-y,m-a)}}\\
\times\exp\bigg(-\int_x^m\dfrac{\frac{s'(u)}{s'(u-y)}\frac{W_{q,1}(u-y,u-a)}{W_q(u-y,u-a)}}{1+\frac{s(u)-s(u-y)}{s'(u-y)}\frac{W_{q,1}(u-y,u-a)}{W_q(u-y,u-a)}}\bigg)\,\diff m\\
=\frac{1}{\cosh[\nu(a-y)]}\bigg(\frac{x-(a-y)}{x}+\frac{\tanh[\nu(a-y)]}{\nu x}\bigg)=\ex_x\{\textup{e}^{-q\sigma_{a-y}}\},
\end{multline}
where the last equality is obtained by substitutions $a\to a-y$ and $y\to 0+$ in \eqref{bes3}. It follows that the law of $C_y^a$ is the same as that of $\sigma_{a-y}$.
\end{proof}

\begin{rmk}
The results of Theorems \ref{thm:samelaw} and \ref{thm:samelaw1} show an nontrivial fact: if $X$ is a drifted Brownian motion or a three-dimensional Bessel process, then for a fixed $y>0$, the law of $\sigma_y$ is the same as $C_{a}^{y+a}$ for any $a>0$. That is, the amount of time the drawdown process $Y$ spends in $[a,a+y]$ until the drawdown time $\sigma_{a+y}$ is the same as the drawdown time $\sigma_y$. 
\end{rmk}

\appendix
\setcounter{secnumdepth}{0}
\section{Appendix}

In the Appendix we provide proofs that have been skipped in the main text. \begin{proof}[Proof of Lemma \ref{property:W_q}]
Most formulas are straightforward and we omit the proofs for them. In the sequel we only prove 
\bq
\lim_{x\downarrow l}\phi_q^-(x)&=&\infty,\,\,\forall q>0,\nn\\
\lim_{q\downarrow 0} W_q(x,y)&=&W_0(x,y),\,\,\forall x,y
\in I.\nn
\eq
First, for $x\in (l,\kappa)$, using \eqref{eq:phi} and monotone convergence theorem, and inaccessibility of $l$ after time 0, we have
\[\lim_{x\downarrow l}\phi_q^-(x)=\lim_{x\downarrow l}\frac{1}{\ex_\kappa\{\mathrm{e}^{-q\tau_x}\}}=\frac{1}{0}=\infty.\]
Secondly, from \eqref{scalefun} we have that for $x\ge y, x,y\in I$, 
\[\frac{\partial}{\partial x}\bigg(\frac{\phi_q^-(x)}{\phi_q^+(x)}\bigg)=-w_q\frac{s'(x)}{(\phi_q^+)^2(x)}\Rightarrow W_q(x,y)= \phi_q^+(x)\phi_q^+(y)\int_y^x\frac{s'(u)}{(\phi_q^+)^2(u)}\diff u.\]
We observe from \eqref{eq:phi} and regularity of $X$ that,  $\phi_q^+(u)$, $u\in [y,x]$ is uniformly bounded (away from 0) for all $q\in [0, q_0]$ for any fixed $q_0>0$:
\bq
0<\ex_y\{\mathrm{e}^{-q_0\tau_\kappa}\}\le \phi_q^+(u)\le \frac{1}{\ex_\kappa\{\mathrm{e}^{-q_0\tau_x}\}}<\infty, \,\,\,\forall u\in [y,x].\nn
\eq
 Moreover, from \eqref{eq:phi} we obtain that, for $u\in[y,x]\subsetneq I$.
\bq
\lim_{q\to 0+}\phi_q^+(u)&=&\left.\begin{dcases}\pr_u\{\tau_\kappa<\infty),\,\,\, &\text{if }u\le \kappa\\
\frac{1}{\pr_\kappa(\tau_u<\infty)},\,\,\, &\text{if }u>\kappa
\end{dcases}\right\}=\lim_{y\downarrow l}\frac{s(u)-s(y)}{s(\kappa)-s(y)}=\beta_1s(u)+\beta_2,\nn
\eq
for some constant $\beta_1,\beta_2$ depending on the behavior of limit $\lim_{y\downarrow l}s(y)$.  By dominated convergence theorem, as $q\downarrow 0$, 
\begin{enumerate}
\item if $\beta_1\neq 0$, 
\begin{align*}W_q(x,y)\to& (\beta_1 s(x)+\beta_2)(\beta_1 s(y)+\beta_2)\int_y^x\frac{s'(u)\diff u}{(\beta_1 s(u)+\beta_2)^2}\\
=&\frac{1}{\beta_1}[(\beta_1s(x)+\beta_2)-(\beta_1s(y)+\beta_2)]=s(x)-s(y)=W_0(x,y);\end{align*}
\item if $\beta_1=0$, then $\beta_2\ge\ex_y\{\mathrm{e}^{-q_0\tau_\kappa}\}>0$, and
\be
W_q(x,y)\to \beta_2^2\int_y^x\frac{s'(u)\diff u}{\beta_2^2}=s(x)-s(y)=W_0(x,y).
\ee
\end{enumerate}
This completes the proof.
\end{proof}

\begin{proof}[Proof of Proposition \ref{propo}]
First we notice that 
\[\pr_m(\tau_n^-<\hat{\sigma}_b\wedge\ep)=\ex_m\{\textup{e}^{-q\tau_n^-}; \tau_n^-<\hat{\sigma}_b\}.\]
To compute the above expectation on the right hand side, we follow the idea of \cite{Lehoczky77}, and partition the interval
$[n,m]$ into $N$ equalength subintervals  with length $\epsilon=\frac{m-n}{N}$.
 In particular,  using the fact that $\pr_m(\tau_m^-=0)=1$ and  continuity of $X$, we have 
 \[\textup{e}^{-q\sum_{i=0}^{N-1}(\tau_{m-(i+1)\epsilon}^{-}-\tau_{m-i\epsilon}^{-})}\ind_{\{ \tau_{m-(j+1)\epsilon}^-<\tau_{m-j\epsilon+b}^+,\,\,\forall  0\le j\le N-1\}}\to \textup{e}^{-q\tau_n^-}\ind_{\{\tau_n^-<\hat{\sigma}_b\}},\,\,\pr_m\text{-a.s.}\]
 as $N\to \infty$. 
 Applying the Lebesgue dominated convergence theorem, the strong Markov property  and continuity of  ${X}$, we obtain that,
 \begin{multline}
\ex_m\{\textup{e}^{-q\tau_n^-}; \tau_n^-<\hat{\sigma}_b\}\\
= \lim_{N\to\infty}\ex_m\{\textup{e}^{-q\sum_{i=0}^{N-1}(\tau_{m-(i+1)\epsilon}^{-}-\tau_{m-i\epsilon}^{-})}; \tau_{m-(j+1)\epsilon}^-<\tau_{m-j\epsilon+b}^+,\,\,\forall  0\le j\le N-1\}\\
=\lim_{N\to\infty}\prod_{i=0}^{N-1} \ex_{m-i\epsilon}\{\textup{e}^{-q\tau_{m-(i+1)\epsilon}^-}\,; \tau_{m-(i+1)\epsilon}^-<\tau_{m-i\epsilon+b}^+\}.\label{laplace:discrete_approx}
\end{multline}

To compute the limit in \eqref{laplace:discrete_approx}. we use  Lemma \ref{fundamentallemma} to obtain that 
\begin{eqnarray}
&&\lim_{N\to\infty}\prod_{i=0}^{N-1}
\ex_{m-i\epsilon}\{\textup{e}^{-q\tau_{m-(i+1)\epsilon}^-}\,; \tau_{m-(i+1)\epsilon}^-<\tau_{m-i\epsilon+b}\}\nn\\
&=&\lim_{N\to\infty}\prod_{i=0}^{N-1}\frac{W_q(m-i\epsilon, m-i\epsilon+b)}{W_q(m-(i+1)\epsilon, m-i\epsilon+b)}\nn\\
&=&\lim_{N\to\infty}\exp\bigg(\log\bigg[1+\sum_{i=0}^{N-1}\frac{W_q(m-i\epsilon, m-i\epsilon+b)-W_q(m-(i+1)\epsilon, m-i\epsilon+b)}{W_q(m-(i+1)\epsilon, m-i\epsilon+b)}\bigg]\bigg)\nn\\
&=&\exp\bigg(\lim_{N\to\infty}\bigg[\sum_{i=0}^{N-1}\frac{W_{q,1}(m-(i+1)\epsilon, m-i\epsilon+b)}{W_q(m-(i+1)\epsilon, m-i\epsilon+b)}\cdot\epsilon +O(\epsilon)\bigg]\bigg)\nn\\
&=&\exp\bigg(\int_n^m\frac{W_{q,1}(v,v+b)}{W_q(v,v+b)}\,\diff v\bigg),\nn
\eq
which completes the proof of \eqref{hlap}. Eq. \eqref{hlap1} can be proved using a similar argument.
\end{proof}

\begin{proof}[Proof of Proposition \ref{thm:occupation_exit}]
We follow the perturbation method in \cite{occupationLevy, LandPenaZhou,LiZhou12}. For $\epsilon>0$ such that $y+\epsilon<b$, we approximate $A_y^{a,b}$ by $A_{y,\epsilon}^{a,b}$:
\[A_{y,\epsilon}^{a,b}:=\sum_{n=1}^\infty (\tau_{y+\epsilon}^{+,n}\wedge\tau_{a,b}-\tau_y^{-,n}\wedge\tau_{a,b}),\]
where $\tau_y^{-,1}:=\tau_y^{-}$, and for $n\ge 1$,
\[\tau_{y+\epsilon}^{+,n}:=\inf\{t\ge\tau_y^{-,n}\,:\,X_t\ge y+\epsilon\},\,\,\, \tau_{y}^{-,n+1}:=\inf\{t\ge\tau_{y+\epsilon}^{+,n}\,:\, X_t\le y\}.\]
Using strong Markov property and the continuity of $X$, we have that 
\bq
&&\ex_y\{\textup{e}^{-qA_{y,\epsilon}^{a,b}-p\tau_b^+}; \tau_b^+<\tau_a^-\}\nn\\
&=&\ex_y\{\textup{e}^{-(q+p)\tau_{y+\epsilon}^+}; \tau_{y+\epsilon}^+<\tau_a^-\}\cdot\ex_{y+\epsilon}\{\textup{e}^{-qA_{y,\epsilon}^{a,b}-p\tau_b^+}; \tau_b^+<\tau_a^-\}\nn\\
&=&\ex_y\{\textup{e}^{-(q+p)\tau_{y+\epsilon}^+}; \tau_{y+\epsilon}^+<\tau_a^-\}\cdot
\left(\ex_{y+\epsilon}\{\textup{e}^{-p\tau_b^+}; \tau_b^+<\tau_y^-\}\right.\nn\\
&&\left.+\ex_{y+\epsilon}\{\textup{e}^{-p\tau_y^-}; \tau_y^-<\tau_b^+\}\ex_y\{\textup{e}^{-qA_{y,\epsilon}^{a,b}-p\tau_b^+};\tau_b^+<\tau_a^-\}\right),\nn
\eq
from which we obtain that,  
\bq
&&\ex_y\{\textup{e}^{-qA_{y,\epsilon}^{a,b}-p\tau_b^+};\tau_b^+<\tau_a^-\}\nn\\
&=&\frac{\ex_{y+\epsilon}\{\textup{e}^{-p\tau_b^+}; \tau_b^+<\tau_y^-\}\ex_y\{\textup{e}^{-(q+p)\tau_{y+\epsilon}^{+}};\tau_{y+\epsilon}^+<\tau_a^-\}}{1-\ex_{y+\epsilon}\{\textup{e}^{-p\tau_y^-}; \tau_y^-<\tau_b^+\}\ex_y\{\textup{e}^{-(q+p)\tau_{y+\epsilon}^{+}};\tau_{y+\epsilon}^+<\tau_a^-\}}\nn\\
&=&\frac{W_p(y+\epsilon,y)}{W_p(b,y)}\frac{\frac{W_{q+p}(y,a)}{W_{q+p}(y+\epsilon,a)}}{1-\frac{W_p(b,y+\epsilon)}{W_p(b,y)}\frac{W_{q+p}(y,a)}{W_{q+p}(y+\epsilon,a)}}.\nn
\eq
The quantity $A_{y,\epsilon}^{a,b}$ measures the time for $X$ to spend below level $y$ and the time to move from $y$ to $y+\epsilon$, but not from $y+\epsilon$ to $y$, until $X$ exits from $(a,b)$. As $\epsilon\to 0+$,  by continuity of $X$, we have $A_{y,\epsilon}^{a,b}\to A_y^{a,b}$, $\pr_x$-a.s.
Using Lebesgue dominated convergence theorem and the continuity of $X$, we have
\bq
\ex_y\{\textup{e}^{-qA_y^{a,b}-p\tau_b^+};\tau_b^+<\tau_a^-\}&=&\lim_{\epsilon\to 0+}\ex_y\{\textup{e}^{-qA_{y,\epsilon}^{a,b}-p\tau_b^+}; \tau_b^+<\tau_a^-\}\nn\\
&=&\frac{s'(y)}{W_{p,1}(y,b)+W_p(b,y)\frac{W_{q+p,1}(y,a)}{W_{q+p}(y,a)}}.\nn
\eq
It follows that, for $x\in(a,y)$, using strong Markov property of $X$, we have
\bq
&&\ex_x\{\textup{e}^{-qA_y^{a,b}-p\tau_b^+};\tau_b^+<\tau_a^-\}\nn\\
&=&\ex_x\{\textup{e}^{-(q+p)\tau_y^+};\tau_y^+<\tau_a^-\}\cdot\ex_y\{\textup{e}^{-qA_y^{a,b}-p\tau_b^+};\tau_b^+<\tau_a^-\}\nn\\
&=&\frac{W_{q+p}(x,a)}{W_{q+p}(y,a)}\frac{s'(y)}{W_{p,1}(y,b)+W_p(b,y)\frac{W_{q+p,1}(y,a)}{W_{q+p}(y,a)}}.\nn
\eq
For $x\in(y,b)$, we similarly have
\bq
&&\ex_x\{\textup{e}^{-qA_y^{a,b}-p\tau_b^+}; \tau_b^+<\tau_a^-\}\nn\\
&=&\ex_x\{\textup{e}^{-p\tau_b^+}; \tau_b^+<\tau_y^-\}+\ex_x\{\textup{e}^{-p\tau_y^-}; \tau_y^-<\tau_b^+)\cdot\ex_y\{\textup{e}^{-qA_y^{a,b}-p\tau_b^+};\tau_b^+<\tau_a^-\}\nn\\
&=&\frac{W_p(x,y)}{W_p(b,y)}+\frac{W_p(b,x)}{W_p(b,y)}\frac{s'(y)}{W_{p,1}(y,b)+W_p(b,y)\frac{W_{q+p,1}(y,a)}{W_{q+p}(y,a)}}.\nn
\eq
Using the similar argument as above, we obtain \eqref{eq:oct_down}.
\end{proof}

\begin{proof}[Proof of Proposition \ref{thm:otbelow}]
We let $\epsilon=\frac{a}{N}$ for a large integer $N>0$. Using Lebesgue dominated convergence theorem, continuity and strong Markov property of $X$, we have
\bq
&&\ex_x\{\textup{e}^{-qB_x^a}; \ol{X}_{\sigma_a}\ge{x+a}\}=\ex_x\{\textup{e}^{-qB_x^a}; \tau_{x+a}^+<\sigma_a\}\nn\\
&=&\lim_{N\to\infty}\prod_{i=0}^{N-1}\ex_{x+i\epsilon}\{\textup{e}^{-qA_x^{x+i\epsilon-a, x+(i+1)\epsilon}}; \tau_{x+(i+1)\epsilon}^+<\tau_{x+i\epsilon-a}^-\}\nn\\
&=&\exp\bigg(\lim_{N\to\infty}\sum_{i=0}^{N-1}\left[\ex_{x+i\epsilon}\{\textup{e}^{-qA_x^{x+i\epsilon-a, x+(i+1)\epsilon}}; \tau_{x+(i+1)\epsilon}^+<\tau_{x+i\epsilon-a}^-\}-1\right]\bigg)\nn\\
&=&\exp\bigg(-\int_x^{x+a}\frac{\frac{s'(u)}{s'(x)}\frac{W_{q,1}(x,u-a)}{W_q(x,u-a)}\,\diff u}{1+\frac{s(u)-s(x)}{s'(x)}\frac{W_{q,1}(x,u-a)}{W_q(x,u-a)}}\bigg).\nn
\eq
Here we used Proposition \ref{thm:occupation_exit} in the last equality.
Moreover, notice that $\ex_x\{\textup{e}^{-qA_x^{u-a,z}}; \tau_{u-a}^-<\tau_{z}^+\}=\ex_x\{\textup{e}^{-qA_x^{u-a,\infty}}; \ol{X}_{\tau_{u-a}^-}<z\}$. It follows that,
\bq
\ex_x\{\textup{e}^{-qB_x^a}; \ol{X}_{\sigma_a}\in(x,{x+a})\}=\int_x^{x+a}\ex_x\{\textup{e}^{-qA_x^{u-a,\infty}}; \ol{X}_{\tau_{u-a}^-}\in \,\diff u\}\nn\\
=\int_x^{x+a}\frac{\partial }{\partial z}\bigg|_{z=u}\ex_x\{\textup{e}^{-qA_{x}^{u-a,z}}; \tau_{u-a}^-<\tau_z^+\}\,\diff u\nn\\
=\int_x^{x+a}\frac{1}{\left(1+\frac{s(u)-s(x)}{s'(x)}\frac{W_{q,1}(x,u-a)}{W_q(x,u-a)}\right)^2}\frac{s'(u)\,\diff u}{W_q(x,u-a)},\nn
\eq
where we used Proposition \ref{thm:occupation_exit} in the last equality.
\end{proof}

\begin{proof}[Proof of Theorem \ref{thm:ddexp}]
Let us denote by 
\be
g:=\inf\{t\ge 0\,:\, X_t=\ol{X}_{\ep}\}.\nn
\ee
Then we have
\be E_y^q=\int_0^g\ind_{\{Y_t>y\}}\,\diff t+\int_g^{\ep}\ind_{\{Y_t>y\}}\,\diff t:=E_y^{q,1}+E_y^{q,2}.\nn\ee
Below we compute the Laplace transforms of $E_y^{q,1}$ and $E_y^{q,2}$ conditioning on $\ol{X}_{\ep}$. More specifically, for $m>x$, we let $\epsilon=\frac{m-x}{N}$  for a large $N>0$. Then we have that 
\bq
&&\frac{\ex_x\{\textup{e}^{-pE_y^{q,1}}\,;\,\ol{X}_{\ep}\in \,\diff m\}}{\diff m}\nn\\
&=&-\frac{\partial}{\partial h}\bigg|_{h=0}\ex_x\{\exp\bigg(-p\int_0^{\tau_m^+}\ind_{\{Y_t>y\}}\,\diff t\bigg);\, \ol{X}_{\ep}\ge m+h\}\nn\\
&=&-\lim_{N\to\infty}\prod_{i=0}^{N-1}\ex_{x+i\epsilon}\{\exp\bigg(-p\int_0^{\tau_{x+(i+1)\epsilon}^+}\ind_{\{X_t<x+i\epsilon-y\}}\,\diff t\bigg); \tau_{x+{i+1}\epsilon}^+<\ep\}\nn\\
&&\times\frac{\partial}{\partial h}|_{h=0}\pr_m(\tau_{m+h}^+\le\ep)\nn\\
&=&\exp\bigg(\lim_{N\to\infty}\sum_{i=0}^{N-1}[\ex_{x+i\epsilon}\{\mathrm{e}^{-p\int_0^{\tau_{x+(i+1)\epsilon}^+}\ind_{\{X_t<x+i\epsilon-y\}}\diff t-q\tau_{x+(i+1)\epsilon}^+}\}-1]\bigg)\cdot\frac{{\phi_q^+}'(m)}{\phi_q^+(m)}\nn\\
&=&\exp\bigg(-\int_x^m\frac{W_{q,2}(u-y,u)+W_{q,1}(u,u-y)\frac{{\phi_{q+p}^+}'(u-y)}{\phi_{q+p}^+(u-y)}}{W_{q,1}(u-y,u)+W_q(u,u-y)\frac{{\phi_{q+p}^+}'(u-y)}{\phi_{q+p}^+(u-y)}}\,\diff u\bigg)\frac{{\phi_q^+}'(m)}{\phi_q^+(m)}.  \nn
\eq
The fourth equality follows from Corollary \ref{cor:onesided}.
On the other hand,
\begin{multline}
\ex_x\{\textup{e}^{-pE_y^{q,2}}|\ol{X}_{\ep}=m\}\\
=\lim_{\epsilon'\to 0+}\dfrac{\ex_m\{\exp\bigg(-p\dint_0^{\ep}\ind_{\{X_t<m-y\}}\,\diff t\bigg)\,;\,\ep<\tau_{m+\epsilon'}^+\}}{\pr_m(\ep<\tau_{m+\epsilon'}^+)}\\
=\lim_{\epsilon'\to0+}\bigg[\dfrac{\pr_m(\ep<\tau_{m-y}^-\wedge\tau_{m+\epsilon'}^+)}{\pr_m(\ep<\tau_{m+\epsilon'}^+)}+\\
\dfrac{\pr_m(\tau_{m-y}^-<\ep\wedge\tau_{m+\epsilon'}^+)}{\pr_m(\ep<\tau_{m+\epsilon'}^+)}\ex_{m-y}\{\mathrm{e}^{-p\int_0^{\ep}\mathbf{1}_{\{X_t<m-y\}}\,\diff t}\,;\,\ep<\tau_{m+\epsilon'}^+\}\bigg]\\
=1+
\lim_{\epsilon'\to0+}\bigg[\dfrac{\pr_m(\tau_{m-y}^-<\ep\wedge\tau_{m+\epsilon'}^+)}{\pr_m(\ep<\tau_{m+\epsilon'}^+)}\ex_{m-y}\{[\mathrm{e}^{-p\int_0^{\ep}\mathbf{1}_{\{X_t<m-y\}}\,\diff t}-1]\,;\,\ep<\tau_{m+\epsilon'}^+\}\bigg].\label{eq:timeafterpeak}
\end{multline}
To get the limit in \eqref{eq:timeafterpeak}, we use Corollary 3.4 of \cite{LiZhou12} to proceed as
\begin{align}
&\lim_{\epsilon'\to0+}\bigg[\dfrac{\pr_m(\tau_{m-y}^-<\ep\wedge\tau_{m+\epsilon'}^+)}{\pr_m(\ep<\tau_{m+\epsilon'}^+)}\ex_{m-y}\{[\mathrm{e}^{-p\int_0^{\ep}\mathbf{1}_{\{X_t<m-y\}}\,\diff t}-1]\,;\,\ep<\tau_{m+\epsilon'}^+\}\bigg]\nn\\
=&\lim_{\epsilon'\to0+}\dfrac{\dfrac{W_q(m+\epsilon',m)}{W_q(m+\epsilon',m-y)}}{1-\dfrac{\phi_q^+(m)}{\phi_q^+(m+\epsilon')}}\nn\\
&\times\bigg\{\dfrac{\dfrac{q}{p+q}W_q(m+\epsilon',m-y)\frac{{\phi_{q+p}^+}'(m-y)}{\phi_{q+p}^+(m-y)}-s'(m-y)+W_{q,1}(m-y,m+\epsilon')}{W_q(m+\epsilon',m-y)\frac{{\phi_{q+p}^+}'(m-y)}{\phi_{q+p}^+(m-y)}+W_{q,1}(m-y,m+\epsilon')}\nn\\
&-\bigg(1-\frac{\phi_q^+(m-y)}{\phi_q^+(m+\epsilon)}\bigg)\bigg\}\nn\\
=&-\frac{s'(m)}{{\phi_q^+}'(m)}\frac{{\phi_{q}^+}'(m-y)+\left[\frac{p}{q+p}\phi_q^ +(m)-\phi_q^+(m-y)\right]\frac{{\phi_{q+p}^+}'(m-y)}{\phi_{q+p}^+(m-y)}}{W_{q,1}(m-y,m)+W_q(m,m-y)\frac{{\phi_{q+p}^+}'(m-y)}{\phi_{q+p}^+(m-y)}}\nn.
\end{align}
It follows that
\begin{multline}
\ex_x\{\textup{e}^{-pE_y^{q,2}}|\ol{X}_{\ep}=m\}\\=\frac{W_{q,2}(m-y,m)+\left[W_{q,1}(m,m-y)-\frac{p}{q+p}s'(m)\right]\frac{{\phi_{q+p}^+}'(m-y)}{\phi_{q+p}^+(m-y)}}{W_{q,1}(m-y,m)+W_q(m,m-y)\frac{{\phi_{q+p}^+}'(m-y)}{\phi_{q+p}^+(m-y)}}\frac{\phi_q^+(m)}{{\phi_q^+}'(m)}.\nn
\end{multline}
The proof is complete after integration with respect to $m$.
\end{proof}

\section*{Acknowledgements}
The author is grateful to Professor Ryozo Miura and Professor Thomas Mikosch for their helpful comments.

\bibliographystyle{plain}

\end{document}